\documentclass[12pt]{article}
\usepackage[centertags]{amsmath}
\usepackage{amsfonts}
\usepackage{amssymb}
\usepackage{amsthm}
\usepackage{xcolor}
\input{amssym.def}
\usepackage{graphicx}
\usepackage{lineno}
\numberwithin{equation}{section}

\newtheorem{Definition}{Definition}[section]
\newtheorem{theorem}[Definition]{Theorem}
\newtheorem{lemma}[Definition]{Lemma}
\newtheorem{proposition}[Definition]{Proposition}

\begin{document}

\title{\Large \bf On $m$-bi-interior ideals of an ordered semigroup}
\author{Susmita Mallick   \\
\footnotesize{Department of Mathematics, Visva-Bharati
University,}\\
\footnotesize{Santiniketan, Bolpur - 731235, West Bengal, India}\\
\footnotesize{mallick.susmita11@gmail.com}}

\date{}
\maketitle

\begin{abstract}
In this article, we concurrently explore the class of $m$-bi-ideals
as well as the class of $m$-interior ideals in an ordered semigroup.
We call these class of ideals as $m$-bi-interior ideals.  Here,
$m$-simple, $m$-regular ordered semigroups and their various
subclasses get characterized in diverse manners as per their
$m$-bi-interior ideals.
\end{abstract}

\section{Introduction and preliminaries}

A semigroup $(S,\cdot)$ with a partial order that is compatible with
the semigroup operation such that for all $a,b\ \textrm{and} \;x\in
S$, $a\leq b$ implies $xa\leq xb$ and $ax\leq bx$, is called an
ordered semigroup and denoted by $(S,\cdot, \leq)$. For an ordered
semigroup $S$ and $H\subseteq S$, denote $(H]:= \{t \in H: t\leq h,
\textrm{for some} \;h\in H\}$. The concept of bi-ideals of
semigroups were first presented by R. A. Good and D. R.
Hughes\cite{R.A. Good} while S. Lajos\cite{lajos} and G. Szasz
{\cite{szasz 77},\cite{szasz 81}} have since developed the idea of
interior ideals in semigroup theory. The bi-interior ideal in
semigroups was a term introduced by M. M. K. Rao\cite{M} to
generalize these two types of ideals. M. Munir\cite{Munir2}
established the idea of $m$-left and $m$-right ideals in semirings,
later this concepts are proposed for the semigroups by A. Shafiq and
M. Munir\cite{Munir3}. In \cite{SM}, we have studied $m$-ideals and
$m$-regularity in an ordered semigroup.

In this paper we introduce the notion of $m$-bi-interior ideals in
an ordered semigroup. Our approach allows one to see a detailed
exposition of $m$-bi-interior ideals in $m$-simple, $m$-regular
ordered semigroup and its different subclasses. we have explored
some inter-relations between $m$-bi-interior ideal and other
$m$-ideals.

\begin{Definition}\cite{M}
A non-empty subset $B$ of a semigroup $S$ is said to be bi-interior
ideal of $S$ if $BSB\cap SBS\subseteq B$.
\end{Definition}

\begin{Definition}\cite{SM} A subsemigroup $L$ of $S$ is called $m$-left($m$-right) ideal of $S$
if $S^{m}L\subseteq L$ ($RS^{m}\subseteq R$) and $(L]=L$($(R]=R$),
for any positive integer $m$ not necessarily $1$.
\end{Definition}
A subsemigroup $I$ of $S$ is called an $m$-two sided ideal or simply
an $m$-ideal of $S$ if it is both $m$-left ideal and $m$-right ideal
of $S$.

\begin{Definition}\cite{SM}
Let $m$ be non-negative integer. An ordered semigroup $S$ is said to
be $m$-left-simple($m$-right-simple) if it does not contain any
proper non trivial $m$-left($m$-right) ideal. An ordered semigroup
$S$ is called $m$-simple if it is both $m$-left and $m$-right
simple.
\end{Definition}

\begin{Definition}\cite{SM}
Let $S$ be an ordered semigroup and $B$ be subsemigroup of $S$, then
$B$ is called $m$-bi-ideal of $S$ if $BS^{m}B\subseteq B$ and
$(B]=B$, where $m\geq 1$ is a positive integer. $m$ is called
bipotency of the bi-ideal $B$.
\end{Definition}

\begin{Definition}
A subsemigroup $I$ of an ordered semigroup $S$ is called
$m$-interior ideal of $S$ if $S^{m}IS^{m}\subseteq I$ and $(I]=I$,
where $m\geq 1$ is a positive integer.
\end{Definition}

\begin{Definition}\cite{SM}
A subsemigroup $Q$ of an ordered semigroup $S$ is called
$m$-quasi-ideal of $S$ if $(S^{m}Q]\cap (QS^{m}]\subseteq Q$ and
$(Q]=Q$, where $m\geq 1$, called quasipotency of quasi-ideal $Q$.
\end{Definition}

\begin{Definition}\cite{SM}
An element $a$ of an ordered semigroup $S$ is called $m$-regular if
$a\leq axa$, for some $x\in S^{m}$. An ordered semigroup $S$ is
$m$-regular if every element of $S$ is $m$-regular that is
$m$-regular if $a\in (aS^{m}a]$, for all $a\in S$.
\end{Definition}

\begin{lemma}\cite{kehayopulu2} \label{AB}
Let $S$ be an ordered semigroup and $A$ and $B$ be subsets of $S$.
Then the following statements hold:

\begin{enumerate}
\item \vspace{-.4cm} $A\subseteq (A]$.
\item \vspace{-.4cm} $((A]]=(A]$.
\item \vspace{-.4cm} If $A\subseteq B$ then $(A]\subseteq (B]$.
\item \vspace{-.4cm} $(A\cap B]\subseteq (A]\cap (B]$.
\item \vspace{-.4cm} $(A\cup B]=(A]\cup (B]$.
\item \vspace{-.4cm} $(A](B]\subseteq (AB]$.
\item \vspace{-.4cm} $((A](B]]=(AB]$.
 \end{enumerate}
\end{lemma}

\begin{lemma}\cite{SM}\label{REGULAR}
Let $S$ be an ordered semigroup. $S$ is $m$-regular if and only if
for every $m$-right ideal $R$ and $m$-left ideal $L$ of $S$,
$(RL]=R\cap L$.
\end{lemma}

\begin{theorem}\cite{SM}
\label{RL} If $S$ be an $m$-regular ordered semigroup then a
non-empty subset $B$ of $S$ is a $m$-bi-ideal of $S$ if and only if
$B=(RL]$ for some $m$-left ideal $L$ and $m$-right ideal $R$ of $S$.
\end{theorem}

\section{\textbf{ $m$-Bi-interior ideals of Ordered Semigroups }}

In this section, we introduce the notion of $m$-bi-interior ideal as
a generalization of $m$-bi-ideal and $m$-interior ideal of an
ordered semigroup and study the properties of $m$-bi-interior ideals
of ordered semigroups, $m$-simple ordered semigroups and $m$-regular
ordered semigroups.

\begin{Definition}
A non-empty subset $B$ of an ordered semigroup $S$ is said to be
$m$-bi-interior ideal of $S$ if $(BS^{m}B]\cap
(S^{m}BS^{m}]\subseteq B$ and $(B]=B$.
\end{Definition}

\begin{Definition}
An ordered semigroup $S$ is said to be $m$-bi-interior-simple
ordered semigroup if $S$ has no $m$-bi-interior ideals other than
$S$ itself.
\end{Definition}

In the following theorem, we mention some important properties.

\begin{proposition}\label{L}
Let $S$ be an ordered semigroup. Then the following statements hold:

\begin{enumerate}

\item \vspace{-.4cm}
Every $m$-left($m$-right) ideal is an $m$-bi-interior ideal of $S$.

\item \vspace{-.4cm} Every $m$-ideal is an $m$-bi-interior ideal.
\item \vspace{-.4cm} If $A$ and $B$ are $m$-bi-interior ideals of $S$ then $A\cap B$(provided the intersection is non-empty) is
an $m$-bi-iterior ideal of $S$.

\item \vspace{-.4cm} Intersection of an $m$-right ideal and an $m$-left ideal of $S$ is
an $m$-bi-interior ideal of $S$.
\item \vspace{-.4cm}
Every $m$-quasi-ideal is an $m$-bi-interior ideal of $S$.

\item \vspace{-.4cm}
Every $m$-bi-ideal of  $S$ is an
$m$-bi-interior ideal of $S$.

\item \vspace{-.4cm}
Every $m$-interior ideal of ordered semigroup $S$ is an
$m$-bi-interior ideal of $S$

\item \vspace{-.4cm}

If $B$ is an $m$-bi-interior ideal of $S$, then $(BS]$ and $(SB]$
are $m$-bi-interior ideals of $S$.

\item \vspace{-.4cm} If $T$ be an $m$-right ideal of $S$ and
$B$ be an $m$-bi-interior ideal of $S$ then $B\cap T$ is an
$m$-bi-interior ideal of $S$.

\end{enumerate}

\end{proposition}

 \begin{proof} $(1)$: Suppose $L$ is an $m$-left ideal of $S$ implies $S^{m}L\subseteq L$ and
$(L]=L$. Now   $(LS^{m}L]\cap (S^{m}LS^{m}]\subseteq
(LSL^{m}]\subseteq(SS^{m}L]\subseteq (S^{m+1}L]\subseteq
(S^{m}L]\subseteq L$. Hence $L$ is an $m$-bi-interior ideal of $S$.

Similarly it can be proved that every $m$-right ideal is an
$m$-bi-interior ideal of $S$.

$(2)$: Since every $m$-ideal is both $m$-left and $m$-right ideal of
an ordered semigroup $S$ then from previous property, it follows
that every $m$-ideal of $S$ is an $m$-bi-interior ideal of $S$.

 $(3)$: Let $A$ and $B$ be $m_{1}$-bi-interior ideal and $m_{2}$-bi-interior ideal respectively of an
ordered semigroup $S$. Then $(AS^{m_{1}}A]\cap
(S^{m_{1}}AS^{m_{1}}]\subseteq A$ and $(A]=A$, $(BS^{m_{2}}B]\cap
(S^{m_{2}}BS^{m_{2}}]\subseteq B$ and $(B]=B$.Let
$m=max\{m_{1},m_{2}\}$ Now $(S^{m}(A\cap B)S^{m}]\subseteq
(S^{m}AS^{m}]\cap (S^{m}BS^{m}]\subseteq (S^{m_{1}}AS^{m_{1}}]\cap
(S^{m_{2}}BS^{m_{2}}]$. Also $((A\cap B)S^{m}(A\cap B)]\subseteq
(AS^{m}A]\cap (BS^{m}B]\subseteq (AS^{m_{1}}A]\cap (BS^{m_{2}}B]$.
Hence $(S^{m}(A\cap B)S^{m}]\cap ((A\cap B)S^{m}(A\cap B)]\subseteq
(S^{m_{1}}AS^{m_{1}}]\cap (S^{m_{2}}BS^{m_{2}}]\cap
(AS^{m_{1}}A]\cap (BS^{m_{2}}B]\subset A\cap B$. Now $(A\cap
B]\subseteq (A]\cap (B]= A\cap B$ implies $(A\cap B]=A\cap B$. Hence
$A\cap B$ is an $m$-bi-interior ideal of $S$.

$(4)$: It follows from above result.

$(5)$: Suppose that $Q$ is an $m$-quasi-ideal of an ordered
semigroup $S$, that is  $(QS^{m}]\cap (S^{m}Q]\subseteq Q$ and
$(Q]=Q$. Now $(QS^{m}Q]\subseteq (QS^{m}] \cap (S^{m}Q]\subseteq Q$
implies $(QS^{m}Q]\cap (S^{m}QS^{m}]\subseteq (QS^{m}Q]\subseteq Q$.
Hence $Q$ is an $m$-bi-interior ideal of $S$.

$(6)$: Suppose that $B$ is an $m$-bi-ideal of an ordered semigroup
$S$ that is  $BS^{m}B\subseteq B$ and $(B]=B$. Now $(BS^{m}B]\cap
(S^{m}BS^{m}]\subseteq (BS^{m}B]\subseteq (B]=B$. Hence $B$ is an
$m$-bi-interior ideal of $S$.

$7$:

Let $I$ is an $m$-interior ideal of $S$. Now $(IS^{m}I]\cap
(S^{m}IS^{m}]\subseteq (S^{m}IS^{m}]\subseteq (I]=I$. Hence $I$ is
an $m$-bi-interior ideal of $S$.

$(8)$: Let $B$ be an $m$-bi-interior ideal of an ordered semigroup
$S$. Now $((BS]S^{m}(BS]]\cap (S^{m}(BS]S^{m}]\subseteq
(BSS^{m}SS]\cap (S^{m}BS^{m+1}]\subseteq (BS^{m+3}]\subseteq (BS]$.
Hence $(BS]$ is an $m$-bi-interior ideal of $S$.

Similarly it can be shown that $(SB]$ is an $m$-bi-interior ideal of
$S$.

$(9)$: Consider $B$ is an $m$-bi-interior ideal of an ordered
semigroup $S$ and $T$ is an $m$-right ideal of $S$. Now $((B\cap
T)S^{m}(B\cap T)]\cap (S^{m}(B\cap T)S^{m}]\subseteq (BS^{m}B]\cap
(TS^{m}T]\cap (S^{m}BS^{m}]\cap (S^{m}TS^{m}]\subseteq (BS^{m}B]\cap
(S^{m}BS^{m}]\cap (TT]\subseteq B\cap (T]\subseteq B\cap T$. Again
$(B\cap T]\subseteq (B]\cap (T]=B\cap T$. Hence $B\cap T$ is an
$m$-bi-interior ideal of $S$.
\end{proof}

But in general, every $m$-bi-interior ideal is not $m$-bi-ideal. In
the next theorem we state the condition for which $m$-bi-ideals and
$m$-bi-interior ideals coincide.
\begin{theorem}Let $S$ be $m$-simple ordered semigroup. Every $m$-bi-interior ideal is
an $m$-bi-ideal of $S$.

\end{theorem}
\begin{proof}
 Let $S$ be an $m$-simple ordered semigroup and $B$ be an $m$-bi-interior
ideal of $S$ then $(BS^{m}B]\cap (S^{m}BS^{m}]\subseteq B$ and
$(B]=B$. Now it can be easily check that $(S^{m}BS^{m}]$ is an
$m$-ideal of $S$. Now $m$-simplicity of $S$ gives $(S^{m}BS^{m}]=S$.
Hence $(BS^{m}B]\cap (S^{m}BS^{m}]\subseteq (BS^{m}B]\cap S\subseteq
B$ which implies $(BS^{m}B]\subseteq B$ that is $BS^{m}B\subseteq
(BS^{m}B]\subseteq B$ and $(B]=B$. Hence $B$ is an $m$-bi-ideal of
$S$.

\end{proof}
\begin{theorem}
Let $S$ be an ordered semigroup then following statements hold :

\begin{enumerate}
\item \vspace{-.4cm} Let $A$ and $C$ be subsemigroups of $S$ and
$B=(AC]$. If $A$ is an $m$-left ideal then $B$ is an $m$-bi-interior
ideal of $S$.
\item \vspace{-.4cm} Let $A$ and $C$ be subsemigroups of $S$ and
$B=(CA]$. If $A$ is an $m$-right ideal then $B$ is an
$m$-bi-interior ideal of $S$.
\end{enumerate}
\end{theorem}

\begin{proof}

$(1)$: Let $A$ and $C$ be subsemigroups of ordered semigroup $S$ and
$B=(AC]$. Suppose $A$ is an $m$-left ideal of $S$ then
$(BS^{m}B]=((AC]S^{m}(AC]]\subseteq (ACS^{m}AC]\subseteq
 (AC]=B$ implies $(BS^{m}B]\cap (S^{m}BS^{m}]\subseteq (BS^{m}B]\subseteq
B$. Hence $B$ is an $m$-bi-interior ideal of $S$.

$(2)$: Let $A$ and $C$ be subsemigroups of ordered semigroup $S$ and
$B=(CA]$. Suppose $C$ is an $m$-right ideal of $S$ then
$(BS^{m}B]=((CA]S^{m}(CA]]\subseteq (CAS^{m}CA]\subseteq (CA]=B$
implies $(BS^{m}B]\cap (S^{m}BS^{m}]\subseteq (BS^{m}B]\subseteq B$.
Hence $B$ is an $m$-bi-interior ideal of $S$.
\end{proof}

\begin{theorem}
Let $S$ be an ordered semigroup. $S$ is m-bi-interior simple ordered
semigroup if and only if $(S^{m}aS^{m}]\cap(aS^{m}a]=S$, for all
$a\in S$.
\end{theorem}

\begin{proof}
Consider $S$ is an m-bi-interior simple ordered semigroup and $a\in
S$. Now $(S^{m}((S^{m}aS^{m}]\cap (aS^{m}a])S^{m}]\cap
(((S^{m}aS^{m}]\cap (aS^{m}a])S^{m}((S^{m}aS^{m}]\cap
(aS^{m}a])]\subseteq (S^{m}(S^{m}aS^{m}]S^{m}]\cap
((aS^{m}a]S^{m}(aS^{m}a]]\subseteq (S^{2m}aS^{2m}]\cap
(aS^{m}SS^{m}SS^{m}a]\subseteq (S^{m}aS^{m}]\cap (aS^{m}a]$. Hence
$(S^{m}aS^{m}]\cap (aS^{m}a]$ is an $m$-bi-interior ideal of $S$
which contradict the fact that $S$ is $m$-bi-interior simple. Hence
$(S^{m}aS^{m}]\cap(aS^{m}a]=S$.

Conversely, suppose that $(S^{m}aS^{m}]\cap(aS^{m}a]=S$, for all
$a\in S$. Let $B$ be an $m$-bi-interior ideal of $S$ and $a\in B$.
Then $S=(S^{m}aS^{m}]\cap(aS^{m}a]\subseteq
(S^{m}BS^{m}]\cap(BS^{m}B]\subseteq B$ which implies $S=B$. Thus $S$
is m-bi-interior simple ordered semigroup.

\end{proof}

\begin{theorem}
An ordered semigroup $S$ is $m$-regular if and only if $B\cap I\cap
L\subseteq (BIL]$ for any $m$-bi-interior ideal $B$,  $m$-ideal $I$
and $m$-left ideal $L$ of $S$.
\end{theorem}
\begin{proof}
Suppose $S$ is an $m$-regular ordered semigroup and $B$, $I$, $L$
are $m$-bi-interior ideal, $m$-ideal and $m$-left ideal of $S$
respectively. Let $a\in B\cap I\cap L$. . Since $S$ is $m$-regular
$a\in (aS^{m}a]\subseteq (aS^{m}aS^{m}aS^{m}aS^{m}aS^{m}a]\subseteq
((aS^{m}aS^{m}a](S^{m}aS^{m}a](S^{m}a]]$.

Now $(aS^{m}aS^{m}a]\subseteq (BS^{m}BS^{m}B]\subseteq
(BS^{m}SS^{m}B]\subseteq (BS^{2m+1}B]\subseteq (BS^{m}B]$. Again
$(aS^{m}aS^{m}a]\subseteq (S^{m}BS^{m}]$. Thus
$(aS^{m}aS^{m}a]\subseteq (BS^{m}B]\cap (S^{m}BS^{m}]\subseteq B$.
Again $(S^{m}aS^{m}a]\subseteq (S^{m}IS^{m}I]\subseteq (I]=I$ and
$(S^{m}a]\subseteq (S^{m}L]\subseteq L$. Hence $a\in
((aS^{m}aS^{m}a](S^{m}aS^{m}a](S^{m}a]]\subseteq (BIL]$ that is
$B\cap I\cap L\subseteq (BIL]$.

Conversely, consider $B\cap I\cap L\subseteq (BIL]$, for any
$m$-bi-interior ideal $B$, $m$-ideal $I$ and $m$-left ideal $L$ of
$S$. Let $R$, $L$ be an $m$-right and $m$-left ideal of $S$. Now, by
assumption  $R\cap L= R\cap S^{m}\cap L\subseteq (RS^{m}L]\subseteq
(RL]$ and $(RL]\subseteq (R\cdot1\cdot1\cdot\cdot S]\subseteq
(RS^{m}]\subseteq (R]=R$. Dually $(RL]\subseteq L$. Hence $R\cap
L=(RL]$. Thus Lemma\ref{REGULAR} yields that $S$ is an $m$-regular.

\end{proof}

\begin{theorem}\label{INT-REG}
Let $S$ be an ordered semigroup. If $S$ is $m$-regular then every
$m$-ideals and $m$-interior ideals co-incide.
\end{theorem}
\begin{proof}
Its very obvious that every $m$-ideal is an $m$-interior ideal.

Now for the converse part, consider an ordered semigroup $S$ such
that $S$ is $m$-regular. Let $I$ be an interior ideal of $S$. Now
$m$-regularity of $S$ infers that $I\subseteq (IS^{m}I]$. Now
$IS^{m}\subseteq (IS^{m}I]S^{m}\subseteq (IS^{m}I](S^{m}]\subseteq
((IS^{m}I]S^{m}I](S^{m}]\subseteq (S^{m+1}IS^{2m+1}\subseteq
(S^{m}IS^{m}]\subseteq (I]=I$. Similarly $S^{m}I\subseteq I$. Hence
$I$ is an $m$-ideal of $S$.
\end{proof}
\begin{theorem}\label{BSB}
Let $S$ be an ordered semigroup and $B$ is an $m$-bi-interior ideal
of $S$. Then $S$ is $m$-regular if and only if $(BS^{m}B]\cap
(S^{m}BS^{m}]=B$.
\end{theorem}
\begin{proof}
First suppose that $S$ is an $m$-regular ordered semigroup and $B$
is an $m$-bi-interior ideal of $S$. Then $(BS^{m}B]\cap
(S^{m}BS^{m}]\subseteq B$ and $(B]=B$. Let $x\in B$ then $x\in
(xS^{m}x]\subseteq (xS^{m}xS^{m}x]$. Hence $(xS^{m}xS^{m}x]\in
(BS^{m}B]\cap (S^{m}BS^{m}]$ which implies $B\subseteq (BS^{m}B]\cap
(S^{m}BS^{m}]$. Hence $(BS^{m}B]\cap (S^{m}BS^{m}]=B$.

Conversely, assume that$ (BS^{m}B]\cap (S^{m}BS^{m}]=B$, for all
$m$-bi-interior ideal $B$ of $S$. Let $B=R\cap L$, where $R$ is an
$m$-right and $L$ is an $m$-left ideal of $S$. Then by Proposition
\ref{L}, $B$ is an $m$-bi-interior ideal of $S$. Now $(BS^{m}B]\cap
(S^{m}BS^{m}]= ((R\cap L)S^{m}(R\cap L)]\cap (S^{m}(R\cap
L)S^{m}]\subseteq (RS^{m}L]\subseteq (RL]$ and $(RL]\subseteq R\cap
L$. Hence following Lemma \ref{REGULAR}, $S$ is an $m$-regular
ordered semigroup.
\end{proof}

\begin{theorem}\label{(Rl]}
Let $B$ be a subsemigroup of an $m$-regular ordered semigroup $S$.
Then for some $m$-right ideal $R$ and $m$-left ideal $L$ of $S$, $B$
can be represented as $B=(RL]$ if and only if $B$ is an
$m$-bi-interior ideal of $S$.
\end{theorem}
\begin{proof}
Let $S$ be an $m$-regular ordered semigroup and $L$ be an $m$-left
and $R$ be an $m$-right ideal of $S$. Now $(BS^{m}B]\subseteq
(RLS^{m}RL]\subseteq (RL]=B$. Hence $(BS^{m}B]\cap
(S^{m}BS^{m}]\subseteq (BS^{m}B]\subseteq B$. Evidently
$(B]=(RL]=B$. Thus $B$ is an $m$-bi-interior ideal of $S$.

Conversely suppose that $B$ is an m-bi-interior ideal of $S$. Now by
Theorem \ref{BSB}, $m$-regularity of $S$ yields $(BS^{m}B]\cap
(S^{m}BS^{m}]=B$. Let $R=(BS^{m}]$ and $L=(S^{m}B]$. It can be
easily check that $R$ and $L$ are $m$-right and $m$-left ideal of
$S$ respectively. Consider $x\in B$. Now $m$-regularity of $S$ gives
$x\in (xS^{m}x]\subseteq (BS^{m}B]$ that is $B\subseteq (BS^{m}B]$.
Now $(RL]\subseteq ((BS^{m}](S^{m}B]]\subseteq (BS^{m}B]$. Again
$(RL]\subseteq ((BS^{m}](S^{m}B]]\subseteq
(BS^{2m}B]=((BS^{m}B]S^{2m}(BS^{m}B]]\subseteq (S^{m}BS^{m}]$. Hence
$(RL]\subseteq (BS^{m}B]\cap (S^{m}BS^{m}]\subseteq B$. Again
 $B\subseteq (BS^{m}B]\subseteq (BS^{m}BS^{m}BS^{m}B]\subseteq \\
 ((BS^{m}]BS^{m}B(S^{m}B]]\subseteq (RSS^{m}SL]\subseteq (RS^{m}L]\subseteq
 (RL]$. Hence $B=(RL]$.
\end{proof}

\begin{center}
\bf{Acknowledgements}
\end{center}
The authors would like to thank the funding agency, the University
Grant Commission (UGC) of the Government of India, for providing
financial support for this research in the form of UGC-CSIR NET-JRF.

\bibliographystyle{amsplain}

\begin{thebibliography}{10}
\baselineskip 5mm







\bibitem{R.A. Good}
R.A. Good and D.R. Hughes, Associated groups for a semigroup, Bull.
Amer. Math. Soc. \textbf{58} (1952) 624-625.






\bibitem{ke92}
N. Kehayopulu,   On completely regular poe-semigroups, \emph{Math.
Japonica} \textbf{37}(1992), 123-130.



\bibitem{Kehayopulu}
N. Kehayopulu, Interior ideals and interior ideal elements in
ordered semigroups, Pure Math. and Appl., (1999) Vol. 10(3),
323-329.

\bibitem{kehayopulu2}
N. Kehayopulu and M. Tsingelis, On left regular ordered semigroups,
Southeast Asian Bull. Math. \textbf{25}(2002),, 609-615.



\bibitem{lajos}
S. Lajos, (m; k; n)-ideals in semigroups. In: Notes on Semigroups
II, Karl Marx Univ. Econ., Dept. Math. Budapest (1976), Vol(1),
12-19.

\bibitem{M}
M. M. K. Rao, Bi-interior Ideals in semigroups, Discussiones
Mathematicae General Algebra and Applications, \textbf{38} (1)
(2018), 69-78.

\bibitem{Munir2}
M. Munir,  and A. Shafiq. \emph{A generalization of bi ideals in
semirings}. Bull. Int. Math. Virt. Inst, \textbf{8}(1)(2018),
123-133.

\bibitem{Munir3}
M. Munir. \emph{On $m$-bi ideals in semigroups}. Bull. Int. Math.
Virt. Inst. \textbf{8}(3)(2018), 461-467.
\bibitem{szasz 77}
G. Szasz, Interior ideals in semigroups. In: Notes on semigroups IV,
Karl Marx Univ. Econ., Dept. Math. Budapest (1977), No. 5, 1-7.

\bibitem{szasz 81}
G. Szasz, Remark on interior ideals of semigroups, Studia Scient.
Math. Hung. \textbf{16} (1981), 61-63.

\bibitem{SM}
Susmita Mallick, On m-quasi-ideals in m-regular ordered semigroups,
arXiv preprint arXiv:2203.10389 (2022).

\bibitem{Y}
Y. Cao,Characterizations of Regular Ordered Semigroup by
Quasi-ideals, Vietnam Journal of Mathematics, \textbf{30(3)} (2002),
239-250.

\end{thebibliography}

\end{document}